\documentclass[12pt]{amsart}
\usepackage{mathrsfs}
\usepackage{pdfsync}
\usepackage{amscd}
\usepackage{amssymb}
\usepackage{amsmath}
\usepackage{amsfonts}
\usepackage{color}
\usepackage{mathrsfs}
\newtheorem{theorem}{Theorem}
\newtheorem{lemma}{Lemma}
\newtheorem{definition}{Definition}

\newtheorem{corollary}{Corollary}
\newtheorem{remark}{Remark}

\def\REQ#1{{\rm (\ref{#1})}}

\def\d{{\rm d}}
\def\e{{\rm e}}

\def\R{{\mathbb R}}
\def\P{{\mathbb P}}
\def\E{{\mathbb E}}
\begin{document}

\title
[Gradient formula for transition semigroup]
{Gradient formula for transition semigroup corresponding to  stochastic equation driven by a system of independent  L\'evy processes
}

\author{Alexei M. Kulik}
\address{Alexei M. Kulik, Faculty of Pure and Applied Mathematics, Wroclaw University of Science and Technology, Wybrze\.ze Wyspia\'nskiego Str. 27
50-370 Wroclaw,  Poland}
\email{oleksii.kulyk@pwr.edu.pl}

\author{Szymon Peszat}
\address{Szymon Peszat,  Institute of Mathematics, Jagiellonian University, {\L}ojasiewicza 6, 30--348 Krak\'ow, Poland}
\email{napeszat@cyf-kr.edu.pl}

\author{ Enrico Priola}
\address{Enrico Priola, Dipartimento di Matematica ''F. Casorati'', University of Pavia,   Via Ferrata, 5 - 27100 Pavia,  Italy}
\email{enrico.priola@unipv.it}

 \thanks{The work of Alexei Kulik   was supported by Polish National Science Center grant 2019/33/B/ST1/02923.  The work of Szymon Peszat    was supported by Polish National Science Center grant 2017/25/B/ST1/02584. The work of Enrico Priola was supported by the grant 346300 for IMPAN from the Simons Foundation and the matching 2015-2019 Polish MNiSW fund.  }  

\begin{abstract}
Let $(P_t)$ be the transition semigroup of the Markov family $(X^x(t))$ defined by SDE
$$
\d X= b(X)\d t + \d  Z, \qquad X(0)=x,
$$
where $Z=\left(Z_1, \ldots, Z_d\right)^*$ is a system of independent real-valued  L\'evy processes. Using the Malliavin calculus we establish
the following gradient formula
$$
\nabla P_tf(x)= \mathbb{E}\, f\left(X^x(t)\right) Y(t,x), \qquad f\in B_b(\mathbb{R}^d),
$$
where the random field $Y$ does not depend on $f$. Moreover, in the important cylindrical $\alpha$-stable case $\alpha \in (0,2)$,   where $Z_1, \ldots , Z_d$ are $\alpha$-stable processes, we are able to prove sharp $L^1$-estimates for $Y(t,x)$. 
 Uniform  estimates
 on $\nabla P_tf(x)$ 
  are also given. 

\end{abstract}

\maketitle

\section{Introduction}


Let $(P_t)$ be the transition semigroup of a Markov family $X=(X^x(t))$ on $\mathbb{R}^d$, that is
\begin{equation}\label{E1}
P_tf(x) =\mathbb{E}\, f(X^x(t)), \qquad f\in B_b(\mathbb{R}^d), \ t\ge 0, x\in \mathbb{R}^d.
\end{equation}
In the paper $X$ is given by the stochastic  differential equation
\begin{equation}\label{E2}
\d X^x(t)= b(X^x(t))\d t + \d Z(t), \qquad X^x(0)=x\in \mathbb{R}^d,
\end{equation}
where $b\colon \mathbb{R}^d\mapsto \mathbb{R}^d$ is a $C^2(\mathbb{R}^d,\mathbb{R}^d)$ and Lipschitz mapping    and
$$
Z(t)=\left(Z_1(t),\ldots, Z_d(t)\right)^*, \qquad t\ge 0,
$$
is a L\'evy process in $\mathbb{R}^d$. We assume that  $Z_j, \ j=1,\ldots, d$,  are independent real-valued  L\'evy processes. We denote by $m_j$ the L\'evy measure of $Z_j$. Recall that
$$
\int_{\mathbb{R}}\, (\xi^2\wedge 1)\,  m_j(\d \xi)<+\infty.
$$
We assume that each $Z_j$ is of purely jump type
 \begin{equation} \label{sw}
Z_j(t)= \int_0^t \int_{\{\xi\in \mathbb{R}\colon |\xi|\ge 1\}} \xi\, \Pi_j (\d s,\d \xi) +  \int_0^t \int_{\{\xi\in \mathbb{R}\colon |\xi|< 1\}} \xi \left[ \Pi_j (\d s,\d \xi) -\d s m_j(\d \xi)\right],
\end{equation}
where $\Pi_j(\d s, \d \xi)$ is a Poisson random measure on $[0,+\infty)\times \mathbb{R}$ with intensity measure $\d sm_j(\d \xi)$.

The main aim of this article is to establish  the following gradient formula
\begin{equation}\label{E3}
\nabla P_tf(x)= \mathbb{E}\, f\left(X^x(t)\right) Y(t,x), \qquad f\in B_b(\mathbb{R}^d),
\end{equation}
where the random field $Y$ does not depend on $f$. The gradient formulae of such type date back to \cite{Bismut}, \cite{Elworty-Li} and are frequently called the \emph{Bismut--Elworthy--Li formulae}.  Note that  \cite{Bismut}  uses an approach based on the Girsanov transformation. On  the other hand  \cite{Elworty-Li} introduces  martingale methods   to derive  formulae like \eqref{E3} in the Gaussian setting; this approach also works for jump diffusions with a non-degenerate Gaussian component (cf. Section 5 in \cite{Priola-Zabczyk}). 

One important consequence of \eqref{E3} is the strong Feller property of the semigroup $(P_t)$, e.g. \cite{Dong-Song}, \cite{Dong-Peng-Song-Zhang}, which in particular motivates our interest in this topic.  Moreover, such gradient formulae allow  the Greeks computations for pay-off functions in mathematical finance.  We refer to  \cite{F} where the authors  apply the Malliavin calculus on the Wiener space to the sensitivity analysis for asset price
dynamics models.

For   L\'evy-driven SDEs with a possibly degenerate Gaussian component, the Bismut--Elworthy--Li formula has been obtained in \cite{Takeuchi} under the assumption on the L\'evy measure to have a density with respect to  Lebesgue measure in $\mathbb{R}^d$; see also \cite{Wang-Xu-Zhang, Zhang} for the   Bismut--Elworthy--Li formula for an SDE driven by a subordinated Brownian motion. In our study, we are focused on the more difficult situation, where the noise is presented by a collection of one-dimensional L\'evy processes, and thus is quite singular.

In plain words, the substantial complication of the problem in our case is that the class of the random vector fields, which are ``admissible'' for the noise in the sense that they allow the integration-by-parts formula, is much more restricted. Namely, in our case only the  ``coordinate axis  differentiability directions'' in $\mathbb{R}^d$ are actually allowed,
while in the case of the L\'evy measure with a density there are no
limitation on these directions.
For the first advances in the Malliavin calculus for L\'evy noises, supported by (singular) collection of curves, we refer to \cite{Leandre}.

In the important cylindrical $\alpha$-stable case  (i.e., when each $Z_j$ is  $\alpha$-stable)  with  $\alpha \in (0,2)$  we obtain  the sharp estimate
\begin{equation}\label{s2}
\sup_{x \in {\mathbb R}^d} {\mathbb E} |Y(t,x)| \le C_T\,  t^{-\frac{1}{\alpha}},\qquad t \in (0,T].
\end{equation}
The  method we use  to obtain \eqref{s2} seems to be of independent interest.
It has two main steps.
The first one is  a bound for   
 $ \mathbb{E}\left\vert Y(t)-Y(t,x)\right\vert$    
where $Y(t)$ corresponds to $Y(t,x)$ when $b=0$ in \eqref{E2}, i.e., $X^x(t) = x + Z(t)$. The second step concerns with $\E |Y(t)|$ (see Section 8). Both steps require sharp estimates and are quite involved (see in particular Sections 6.2 and 8). 
 Formula \eqref{s2}  implies the bound  ($\|\cdot \|_\infty$ stands for the supremum norm)
\begin{gather} \label{w22} 
\| \nabla P_tf\|_\infty := \sup_{x \in {\mathbb R}^d} |\nabla P_tf(x)| \le C_Tt^{-\frac{1}{\alpha}} \| f \|_{\infty}, \qquad f\in B_b(\mathbb{R}^d), \qquad t \in (0,T].
\end{gather}
 It seems that  when $0 < \alpha \le 1$ also  estimate \eqref{w22} is new; it  cannot be obtained by a perturbation argument which is available when $\alpha >1$.  In fact we will establish  \REQ{w22} for any process $Z$  with small jumps similar to $\alpha$-stable process.
 Recall that  estimates like \eqref{w22}  for $\alpha >1$ hold even in some non-degenerate multiplicative cases (see Theorem 1.1  in \cite{Kulczycki-Ryznar}; in such result    the Lipschitz   case
 $\gamma =1$ requires   $\alpha >1$).
We expect that our approach should also work for SDEs with  multiplicative cylindrical noise; such an extension is a subject of our ongoing research.

Let us mention that from the analytical point of view we are concerned with the gradient estimates of the solution to the following  equation with a  non-local operator 
\begin{align*}
\frac{\partial u}{\partial t}(t,x) &= \langle a + b(x), \nabla u(t,x)\rangle \\
&\quad + \sum_{j=1}^d \int_{\mathbb{R}}\left (u(t,x+ \xi e_j)- u(t,x)-\chi_{\{ |\xi |\le 1\}} \xi \frac{\partial u}{\partial x_j}(x)\right)m_j(\d \xi), \;\;\; t>0,
\end{align*}
$u(0,x) =f(x)$, where $e_j, j=1, \dots, d$,  is the canonical basis of $\mathbb{R}^d$.
\section{Main result}


Let $Q_tf(x)=\mathbb{E}\, f(Z^x(t))$ be the transition semigroup corresponding to the L\'evy proces   $Z^x(t)= x + Z(t)$.  The proof of the following  theorem concerning BEL formulae for $(P_t)$ and $(Q_t)$  is postponed to Section \ref{S6}.
\begin{theorem}\label{T1}
Let $P=(P_t)$ be given by $\REQ{E1}$, $\REQ{E2}$. Assume that:
\begin{itemize}
\item[(i)]  $b\in C^2(\mathbb{R}^d, \mathbb{R}^d)$ has bounded derivatives $\frac{\partial  b_i}{\partial \xi_j}$, $\frac{\partial ^2 b_i}{\partial \xi_j\partial \xi_k}$,  $i,j,k=1,\ldots, d$.
\item[(ii)] There is a $\rho >0$ such that
$$
\liminf_{\varepsilon \downarrow 0} \varepsilon ^{\rho} m_j\{|\xi |\ge \varepsilon\}\in (0,+\infty], \quad j=1,\ldots, d.
$$
\item[(iii)] There is a $\delta>0$ such that each $m_j$ restricted to the interval $(-\delta, \delta)$  is  absolutely continuous with respect to Lebesgue measure. Moreover, the density $\rho_j= \frac{\d m_j}{\d \xi}$ is of class $C^1((-\delta,\delta)\setminus \{0\})$  and  there exists a   $\kappa >1     $  such that for all $j$,
\begin{align}\label{E4}
\int_{-\delta}^{\delta} |\xi|^\kappa \rho_j(\xi)\d \xi  &<+\infty, \\ \label{E5} 
\int_{-\delta} ^{\delta} |\xi|^{2\kappa} \left( \frac{\rho'_j(\xi)}{\rho_j(\xi)}\right)^2 \rho_j(\xi)\d \xi&<+\infty
,\\
\int_{-\delta} ^{\delta} |\xi|^{2\kappa-2}  \rho_j(\xi)\d \xi&<+\infty. \label{E6}
\end{align}
\end{itemize}
Then  there are  integrable random fields $Y(t)= \left(Y_1(t), \ldots, Y_d(t)\right)$, and  $Y(t,x)= \left( Y_1(t,x),\ldots, Y_d(t,x)\right)$, $t> 0$, $x\in \mathbb{R}^d$,  such that for any $f\in B_b(\mathbb{R}^d)$, $t>0$, $x\in \mathbb{R}^d$,
$$
\nabla Q_tf(x)= \mathbb{E}\, f(Z^x(t)) Y(t)
$$
and the Bismut--Elworthy--Li formula  \eqref{E3} for $(P_t)$ holds.  Moreover,  for any $T >0$ there is an independent   of $t \in (0,T]$  and $x$ constant $C$ such that
\begin{equation}\label{E7}
\mathbb{E}\left( \vert Y(t)\vert + \vert Y(t,x)\vert \right) \le Ct^{-\frac{\kappa}{\rho}+\frac{1}{2}}.
\end{equation}
Finally, for any $T>0$ and  $\varepsilon \in (0,1/2)$ there is a constant $C_\varepsilon$ such that for all  $t \in (0,T]$  and $x$,
\begin{equation}\label{E8}
\mathbb{E}\left\vert Y(t)-Y(t,x)\right\vert  \le C_\varepsilon t^{-\frac{\kappa}{\rho}+ 3/2}.
\end{equation}
\end{theorem}
\begin{remark}
{\rm  Note that, what is expected,  the rate $-\frac{\kappa}{\rho}+ 3/2$ depends only on the small jumps of $Z$. } 
\end{remark}
\begin{remark}\label{R1}
{\rm In fact we have formulae for the fields  appearing in Theorem \ref{T1}. Namely,
\begin{equation}\label{E9}
\begin{aligned}
Y_j (t)&=  \sum_{k=1}^d \left[ A_{k,j}(t) D_{k}^*\mathbf{1}(t)-D_{k}A_{k,j}(t)\right],\\
Y_j (t,x)&=  \sum_{k=1}^d \left[ A_{k,j}(t,x) D_{k}^*\mathbf{1}(t)-D_{k}A_{k,j}(t,x)\right],
\end{aligned}
\end{equation}
where:
\begin{itemize}
\item the matrix-valued random fields  $A(t)= \left[ \mathbb{D} Z(t)\right] ^{-1}$ and $A(t,x)=[A_{k,j}(t,x)]\in M(d\times d)$ are  given by
\begin{equation}\label{E10}
\begin{aligned}
A(t)&= \left[ \mathbb{D} Z(t)\right] ^{-1}, \\
A(t,x)&= \left[ \mathbb{D} X^x(t)\right] ^{-1}\nabla X^x(t), \;\;\;\; \mathbb{P-}{\text a.s},
\end{aligned}
\end{equation}
\item    $\mathbb{D} Z(t)$ and $\mathbb{D} X^x(t)$ are the Malliavin derivatives (see Section \ref{S3} and formulae \eqref{E18} and \eqref{23}) of $Z(t)$ and  $X^x(t)$ respectively,  with respect to the field $V= (V_1,\ldots, V_d)$,
\begin{equation}\label{E11}
V_j(t,\xi)= \phi_{\delta}(\xi_j)\psi_{{\delta}}(t) = V_j(t,\xi_j),
\end{equation}
$\psi_{{\delta}}\in C^\infty(\mathbb{R})$  and $\phi_\delta \in C^\infty(\mathbb{R}\setminus\{0\})$    are non-negative functions   such that
\begin{equation}\label{E12}
\psi_{{\delta}}(r)=\begin{cases}
0&\text{if $|r|\ge {\delta}$,}\\
1&\text{if $|r|\le {\delta}/2$}
\end{cases}, 
\qquad
\phi_\delta(r)= |r|^{\kappa}\psi_{\delta}(r),
\end{equation}
 with constants $ {\delta}>0$ and   $\kappa$ appearing in assumption $(iii)$ of Theorem \ref{T1}, 
  \item $\nabla X^x(t)$ is the derivative in probability  of $X^x$ with respect to the initial condition $x$,
\item $D_{k}^*\mathbf{1}(t)$ is the adjoint derivative operator calculated on the constant function $\mathbf{1}$, see Section \ref{S3}, Lemma \ref{L3}. We note that the matrix $A(t)$ is diagonal with entries
$$
\left( \int_0^t \int_{\mathbb{R}} V_j(s,\xi_j)\Pi_j(\d s,\d \xi _j)\right)^{-1}.
$$
 \end{itemize}
}
\end{remark}

\begin{remark}\label{R4} {\rm  The fields $Y(t)$ and   $Y(t,x)$ are  not uniquely determined by the BEL formulae. However their conditional expectations $\mathbb{E}\left( Y(t)\vert Z^x(t)\right)$  and  $\mathbb{E}\left( Y(t,x)\vert X^x(t)\right)$ are  uniquely determined. On the other hand, $\mathbb{E}\left\vert Y(t)\right\vert$ and $\mathbb{E}\left\vert Y(t,x)\right\vert$ may depend on the choice of the fields.
}
\end{remark}

Estimate \eqref{E7} implies   new  uniform gradient estimates
\begin{equation}
 \label{grad1}
\left \| \nabla P_t f\right\| _\infty   \le C_{T, \epsilon}  t^{-\frac{\kappa}{\rho}+ 1/2} \left\| f
 \right\| _\infty, \qquad  t \in (0,T],\ \  f \in B_b({\mathbb R}^d).
\end{equation}
Although   \eqref{grad1} is quite general, it is not sharp in the relevant cylindrical  $\alpha$-stable case with  $\alpha \in (0,2)$.
In such  case  $\rho =\alpha$ and $\kappa$ is any real number satisfying $\kappa >  1+ \frac \alpha 2$. Therefore we only get
 that for any $\varepsilon>0$ and $T<+\infty$ there is a constant $C_{\varepsilon,T}$ such that f or any $f\in B_b(\mathbb{R}^d)$,
\begin{gather}\label{ww}  
\left \| \nabla P_t f\right\| _\infty   \le C_{\varepsilon,T} t^{-\frac{1}{\alpha}  -\varepsilon } \left\| f \right\| _\infty, \qquad t \in (0,T]. 
\end{gather}
  We will improve the previous estimate in Section \ref{vv} by considering $\varepsilon =0$. To this purpose we will also use the next remark.
 
\begin{remark}\label{qq}
 {\rm 
 Our main theorem provides also estimate  \REQ{E8} for  $\mathbb{E}\left\vert Y(t,x)- Y(t)\right\vert$.
 This can be  useful. Indeed if for some  specific   L\'evy processes $Z_j$ we have
\begin{equation}\label{s22}
\mathbb{E}\left\vert Y(t)\right\vert \le C_Tt^{-\eta}, \qquad  t \in (0,T]
\end{equation}
or even if $\mathbb{E}\left\vert \mathbb{E}\left( Y(t)\vert X^x(t)\right) \right\vert \le C_Tt^{-\eta}$ for some $\eta$ such that
$$
\frac {\kappa}{\rho} - 1 < \eta \le \frac {\kappa}{\rho}-\frac 12,
$$
 where $\kappa $ verifies our assumptions,
then  we can improve \eqref{E7} and get, for $t \in (0,T]$,
\begin{equation}\label{E71}
\mathbb{E} \vert Y(t,x)\vert  \le C_T't^{-\eta},\qquad t \in (0,T].
\end{equation}
By \eqref{E71} one deduces
$$
\left \| \nabla P_t f\right\| _\infty   \le C_T' t^{-\eta } \left\| f \right\| _\infty, \qquad t \in (0,T].
$$
In particular when $Z_j$ are independent real $\alpha$-stable processes we will get in Section \ref{vv} the crucial estimate
\begin{equation}\label{E7112}
\mathbb{E} \left\vert Y(t)\right\vert  \le C _Tt^{-\frac{1}{\alpha}}, \qquad t \in (0,T].
\end{equation}
Combining \eqref{E8} with \eqref{E7112} we deduce  in the cylindrical $\alpha-$stable case
\begin{equation}\label{E711}
\mathbb{E} \left\vert Y(t,x)\right\vert  \le C_T't^{-\frac{1}{\alpha}}, \qquad t \in (0,T],
\end{equation}
(where $C'_{T}$ is independent of $x$ and $t$)
 and the sharp gradient estimate
\begin{gather*}
\left \| \nabla P_t f\right\| _\infty   \le C_T' t^{-\frac{1}{\alpha} } \left\| f \right\| _\infty, \qquad  t \in (0,T],\ \  \alpha \in (0,2).
\end{gather*}
}
\end{remark}

\bigskip As mentioned in the introduction a difficulty of the proof of   Theorem  \ref{T1}   is also to show that the Malliavin derivative of the solution $\mathbb{D} (X^x(t))$ in the direction to a suitable random field $V$ is invertible and the inverse is integrable with sufficiently large power. The idea (see the proof of our Lemma \ref{L5}) is to show that $\mathbb{D} (X^x(t))\approx Z^V(t)$,  where $Z^V(t)$ is a diagonal matrix with the terms $\int_0^t \int_{\mathbb{R}^d} V_j(s,\xi)\Pi_j(\d s ,\d \xi)$ on diagonal. Therefore the integrability of $\left(\mathbb{D} (X^x(t))\right)^{-1}$ follows from the known fact, see Section \ref{S5} that
$$
\mathbb{E}\left[ \int_0^t \int_{\mathbb{R}}\psi(\xi) \Pi_j(\d s,\d \xi)\right]^{-q}\le C(q,T)\, t^{-\frac{\kappa q}{\rho}}, \qquad \forall \, q\in (1,+\infty).
$$
 On the other hand, 
 several technical difficulties arise in proving  the   sharp  bounds for   
 $ \mathbb{E}\left\vert Y(t)-Y(t,x)\right\vert$    and 
 $\E |Y(t)|$.

Finally, we mention that an attempt to prove \eqref{E3} has been done in \cite{BHR} by the martingale approach used in   \cite{Takeuchi}  (see, in particular, Lemma A.3 in \cite{BHR}). However the  BEL formula in   \cite{BHR}  does not seem to be  correct,  since  there is  a gap in the proof, passing  from formula (48) to (49) in page 1450 of \cite{BHR}, which consists in an undue application of the chain rule. It seems that the complication here is substantial, and it is difficult to  adapt directly the approach used in \cite{Takeuchi} to the current setting, where because of singularity of the noise it is hard to guarantee invertibility of the Malliavin derivative w.r.t. one vector field. Exactly this crucial point is our reason to use a matrix-valued Malliavin derivative of the solution w.r.t. a vector-valued field $V= (V_1,\ldots, V_d)$.

\section{Malliavin calculus}\label{S3}


In this section we adopt in a very direct way the classical concepts and results of Bass and Cranston \cite{Bass-Cranston} and Norris \cite{Norris} to the case of  $Z=(Z_1,\ldots, Z_d)^*$ being  a  L\'evy process in $\mathbb{R}^d$ with independent   coordinates $Z_j$.  For more information on Malliavin calculus for jump processes we refer the reader to the book of Ishikawa \cite{Ishikawa} (see also \cite{BC2011} and the references therein).

We assume that $Z=\left(Z_1,\ldots, Z_d\right)^*$ is defined on a probability space $(\Omega,\mathcal{F},\mathbb{P})$. By the L\'evy--It\^o decomposition
$$
Z(t)= \int_0^t \int_{\mathbb{R}^d} \xi\, \overline {\Pi}(\d s, \d \xi),
$$
where   $\Pi$ be the Poisson random measure on $E:= [0,+\infty)\times \mathbb{R}^d$  with intensity measure $\d t\mu(\d \xi)$,
\begin{align*}
\overline{\Pi}(\d s ,\d \xi)&:= \widehat {\Pi}(\d s ,\d \xi) \chi_{\{|\xi |\le 1\}}+ {\Pi}(\d s ,\d \xi)\chi_{\{|\xi |> 1\}},\\
\widehat {\Pi}(\d s ,\d \xi)&:= \Pi(\d s,\d \xi)-\mu(\d \xi)\d s.
\end{align*}
Moreover, as the coordinates of $Z$ are independent,
\begin{equation}\label{E13}
\begin{aligned}
\mu(\d \xi) &:= \sum_{j=1}^d \mu_j(\d \xi),\\
\mu_j(\d \xi) &:= \delta_0(\d \xi_1)\ldots \delta_0(\d \xi_{j-1})m_j(\d \xi _j) \delta_0(\d \xi _{j+1})\ldots \delta_0(\d \xi_d),
\end{aligned}
\end{equation}
where $\delta_0$ is the Dirac $\delta$-function, and $m_j(\d \xi _j)$ is the L\'evy measure of $Z_j$. Note that
$$
\Pi(\d t,\d \xi)=  \sum_{j=1}^d \Pi_j(\d t,\d \xi),
$$
where $\Pi _j$ are independent Poisson random measures each on  $[0,+\infty)\times  \mathbb{R}^d$   with the intensity measure $\mu_j$  (we use the same symbol as for the one-dimensional $\Pi_j(\d s, \d \xi)$ appearing in \eqref{sw}  when no confusion may arise).

Consider the filtration
$$
\mathfrak{F}_t={\delta} \left(\Pi([0,s]\times \Gamma)\colon 0\le s\le t,\ \Gamma \in \mathcal{B}(\mathbb{R}^d)\right), \qquad t\ge 0.
$$

Poisson random field $\Pi$ can be treated as a random element in the space $\mathbb{Z}_+(E)$ of integer-valued  measures  on $(E,\mathcal{B})$ with the {${\sigma}$-field} $\mathcal{G}$ generated by the family of functions
$$
\mathbb{Z}_+(E)\ni \nu \mapsto \nu(A)\in \{0,1,2,\ldots, +\infty\}, \qquad A\in \mathcal{B}.
$$
\begin{definition}
{\rm Let $p\in (0,+\infty)$. We call a random variable $\Psi\colon \Omega\mapsto \mathbb{R}$ an $L^p$-functional of $\Pi$ if there is a sequence of bounded measurable functions $\varphi _n\colon \mathbb{Z}_+(E)\mapsto \mathbb{R}$ such that
\begin{equation}\label{E14}
\lim_{n\to +\infty} \mathbb{E}\left \vert \Psi -\varphi_n(\Pi)\right\vert ^p=0.
\end{equation}
A random variable $\Psi\colon \Omega\mapsto \mathbb{R}$ is called an $L^0$-functional of $\Pi$ if, instead of \eqref{E14}, the convergence in probability holds
\begin{equation}\label{E15}
\varphi_n(\Pi) \stackrel{(\mathbb{P})}{\to}  \Psi.
\end{equation}
}
\end{definition}
The space of all $L^p$-functionals of $\Pi$ is  denoted by $L^p(\Pi)$. Note that for $p\geq 1$,  $L^p(\Pi)$ is a Banach space with the norm
$\|\Psi\|_{L_p}= \Big(\mathbb{E}\left \vert \Psi \right\vert ^p\Big)^{1/p}$,  and for $p\in (0,1)$,  $L^p(\Pi)$ is a Polish space with the metric $\rho_{L_p}(\Phi, \Psi)=\mathbb{E}\left \vert \Phi-\Psi \right\vert ^p$.

 Assume now that $V=(V_1,\ldots, V_d)\colon [0,+\infty)\times \mathbb{R}^d\mapsto \mathbb{R}^d$ is a random field given by \REQ{E11} and \REQ{E12}. The parameter  ${\delta}$ appearing in \REQ{E12}  will be specified later. Define transformations $\mathcal{Q}^\varepsilon _k$, $\varepsilon >0$ and $k=1,\ldots,d$, $\mathcal{Q}^{\varepsilon}_k\colon \mathbb{Z}_+(E)\mapsto \mathbb{Z}_+(E)$ as follows
$$
\mathcal{Q}^{\varepsilon}_k \left( \sum_{j}\delta_{\tau^j , \xi^j}\right) = \sum_{j}\delta_{\tau^j, \xi^j +  \varepsilon V_k(\tau^j, \xi^j_k)e_k},
$$
where $e_k, k=1, \dots, d$,  is the canonical basis of $\mathbb{R}^d$.

Now let  $\Psi\in L^0(\Pi)$. Write
$$
\mathcal{Q}^{\varepsilon}_k\Psi = (\mathbb{P})-\lim_{n\to +\infty} \varphi_n(\mathcal{Q}^{\varepsilon}_k(\Pi)),
$$
where $\varphi_n\colon \mathbb{Z}_+(E) \mapsto \mathbb{R}$ are such that \REQ{E15} holds true.  It follows  from Lemma \ref{L2} below that $\mathcal{Q}^{\varepsilon}_{k}\Psi$ is well defined.
\begin{definition}{\rm  We call $\Psi\in L^0(\Pi)$,  \emph{differentiable} { (with respect to the field $V= (V_1,\ldots, V_d)$)}  if there exist limits in probability 
 $$
 D_k\Psi=(\mathbb{P})-\lim_{\varepsilon\to 0}\frac{1}{\varepsilon}\left(\mathcal{Q}^{\varepsilon}_k(\Psi)-\Psi\right), \quad k=1,\ldots, d.
 $$}
\end{definition}
If $\Psi\in L^0(\Pi)$ is differentiable then we call 
$$
\mathbb{D}\Psi= \left(D_1\Psi, \ldots, D_d\Psi\right)
$$ 
the \emph{Malliavin derivative of $\Psi$}.

The proof of the following chain rule is standard and left to the reader.
\begin{lemma}\label{L1}
Assume that $\Psi_1, \ldots,\Psi_m$ are differentiable functionals of $\Pi$. Then for any $f\in C^1_b(\mathbb{R}^m)$ the variable $f\left( \Psi_1, \ldots,\Psi_m\right)$ is differentiable and
 \begin{equation}\label{E16}
{D}_k f\left( \Psi_1, \ldots,\Psi_m\right)= \sum_{j=1}^m \frac{\partial f}{\partial x_j}\left( \Psi_1, \ldots,\Psi_m\right) {D}_k  \Psi_j, \qquad k=1, \ldots,d.
\end{equation}
\end{lemma}

Let $\rho_k=\frac{\d m_k}{\d x}$ be the density of the L\'evy measure $m_k$ restricted to $(-\delta,\delta)\setminus\{0\} \subset \mathbb{R}$.  Given $\varepsilon \in [-1,1]$   and $k=1,\ldots, d$, define
\begin{align*}
\lambda ^{\varepsilon}_k(t,\xi_k)&:=
\begin{cases}\left(  1 +\varepsilon \frac{\d  V_k}{\d \xi_k}(t,\xi_k)\right) \frac{\rho_k(\xi_k+\varepsilon V_k(t,\xi_k))}{\rho_k(\xi_k)}&, \text{if $\xi_k\in (-\delta,\delta)\setminus \{0\}$,}\\
1&, \text{otherwise},
\end{cases}
\\
  \Lambda^\varepsilon _{k}(t,\xi_k)&:=  \lambda ^{\varepsilon }_k  (t, \xi_k)-1-  \log \lambda ^{\varepsilon }_k (t, \xi_k),
\end{align*}
and
\begin{gather*}
M^{\varepsilon }_k(t):= \exp\Big\{ \int_0^t \int_{\mathbb{R}^d} \log \lambda ^{\varepsilon }_k (s, \xi_k)\widehat{\Pi}_k(\d s,\d \xi)    -  \int_0^t \int_{\mathbb{R}^d} \Lambda ^{\varepsilon }_k  (s, \xi_k) \mu_k(\d \xi) \d s \Big \},
\end{gather*}
where $\mu_k$ is defined in \eqref{E13} and
$$
\widehat\Pi_k ( \d s, \d \xi):=\Pi_k(\d s, \d \xi)-  \mu_k(\d \xi)\d s.
$$

We will need the following result (see e.g. \cite{Norris} or \cite{Ivanenko-Kulik}).
\begin{lemma}\label{L2}
The process $M^{\varepsilon }_k $  is a martingale and for all $T\ge 0$, and  $r\in \mathbb{R}$,  $\mathbb{E} \left[ M^{\varepsilon}_k(T)\right]^r<+\infty$. Let $T\in (0,+\infty)$. Then, under the probability $\d \mathbb{P}^\varepsilon = M^{\varepsilon }_k(T)\d \mathbb{P}$,  $\mathcal{Q}^{\varepsilon }_k(\Pi)$  restricted to $[0,T]\times \mathbb{R}^d$  is a Poisson random measure with intensity ${ \mu_k (\d \xi)\d s }$.
\end{lemma}

The following lemma provides an integration by parts formula for the derivative $D_k$.  For the completeness we repeat some elements of a proof from \cite{Norris}.

\begin{lemma}\label{L3}
For any  $1 \le q\le 2$ and $t\in (0,+\infty)$,  the random variable
\begin{equation}\label{E17}
D_{k}^*\mathbf{1}(t):= - { \int_0^t \int_{(-\delta,\delta) \times {\mathbb R}^{d-1} } } \frac{\frac {\d }{\d  \xi_k} \left(V_k(s,\xi_k)\rho_k(\xi _k)\right)} {\rho_k(\xi _k)}\widehat\Pi_k ( \d s, \d \xi)
\end{equation}
is $q$-integrable. Assume that  { $p \ge 2$ }  and that ${ \Phi}  \in L^p(\Pi)$ is differentiable and  $\mathfrak{F}_t$-measurable. Then $
\mathbb{E} D_k\Phi = \mathbb{E} \Phi D_k^*\mathbf{1}(t)$.
\end{lemma}
\begin{proof} Note that the process $D_{k}^*\mathbf{1}(t)$ is well-defined and $q$-integrable thanks to \REQ{E5}. By Lemma \ref{L2} we have
$$
\frac{\d }{\d \varepsilon} \mathbb{E}\left( \mathcal{Q}^{\varepsilon}_ k  \Phi \right) M^{\varepsilon }_k (t)=0.
$$
Thus
$$
0 = \mathbb{E} \left[ D_{k} \Phi  M^{0}_k (t)+ \Phi  R(t)\right]= \mathbb{E} \left[ D_{k}\Phi + \Phi  R(t)\right],
$$
where
$$
R(t):= \frac{\d }{\d \varepsilon} M^{\varepsilon }_k  (t)\vert_{\varepsilon =0}.
$$
Consequently, we need to show that $D_{k}^*\mathbf{1}(t)= -R(t)$.

Since
$$
M^{\varepsilon }_k  (t)= \exp\left\{ \int_0^t \int_{ { \mathbb{R}^d } } \log \lambda^{\varepsilon }_{k} (s,\xi_k)\widehat \Pi_k(\d s, \d \xi ) - \int_0^t \int_{ {  (-\delta, \delta) \times {\mathbb R}^{d-1}  }   }   \Lambda^{\varepsilon }_k (s,\xi_k )  {\mu_k (\d \xi)} \d s\right\},
$$
we have
\begin{align*}
R(t)&=
 \int_0^t \int_{ { \mathbb{R}^d }}  \frac{\frac{\d }{\d \varepsilon} \lambda^{\varepsilon}_{k} (s,\xi_k)}{\lambda^{\varepsilon }_k (s,\xi_k
 )}\vert_{\varepsilon =0} \widehat{\Pi}_k(\d s, \d \xi)
 - \int_0^t
 \int_{ { (-\delta, \delta) \times {\mathbb R}^{d-1}  }   }  \frac{\d }{\d \varepsilon} \Lambda^{\varepsilon }_{k} (s,\xi_k )\vert_{\varepsilon =0} { \mu_k} (\d \xi) \d s.
\end{align*}
Finally
 {
\begin{align*}
\frac{\d }{\d \varepsilon}\lambda^{\varepsilon }_{k} (s,\xi_k)\vert_{\varepsilon =0}
&= \frac{\partial  V_k}{\partial x_k} (t,\xi_k) + \frac{\rho_k'(\xi_k)}{\rho_k(\xi_k)} V_k(s,\xi_k)
= \frac{\frac {\d  }{\d  \xi_k} \left(V_k(s,\xi_k)\rho_k(\xi_k)\right)} {\rho_k(\xi_k)}.
\end{align*}
}
\end{proof}

\section{Malliavin derivative of $X^x$}\label{S4}

Let $X^x(t)=\left[ X_1^x(t), \ldots, X_d^x(t)\right]^*\in \mathbb{R}^d$ be the value of the solution at time $t$. We use the convention that the vectors in $\mathbb{R}^d$ are columns, and the derivatives (gradients) are rows.   Using  the chain rule (see Lemma \ref{L1}) it is easy to check that for any $p\ge 1$, each of its coordinate  is  a $p$-differentiable functional of $\Pi$ and the $d\times d$-matrix  valued process $\mathbb{D} X^x(t)= \left[ D_jX^x_i(t)\right]$ satisfies the following random ODE
\begin{equation}\label{E18}
\d  \mathbb{D} X^x(t)= \nabla b(X^x(t))\mathbb{D} X^x(t)\d t + \d Z^V(t),\qquad  \mathbb{D} X^x(0)=0
\end{equation}
(cf. Section 5 in \cite{Bass-Cranston}) where $Z^V(t)= \left[Z^V_{ij}(t)\right]$, $t \ge 0,$ is a $d\times d$-matrix valued process
\begin{equation}\label{E19}
Z^V_{ii}(t) := \int_0^t \int_{\mathbb{R}}V_i(s,\xi)\Pi_i(\d s,\d \xi), \qquad Z^V_{ij}(t)=0\quad \text{if $j\not = i$}.
\end{equation} 
Note that $\int_{\mathbb{R}} \vert V_i(t,\xi)\vert m_i(\d\xi)<+\infty$ thanks to  \REQ{E4}, and therefore process $Z^V$ is well defined and $q$-integrable for any $q\in [1,+\infty)$. Clearly we have:
\begin{equation}\label{23}
 \mathbb{D} Z(t) = Z^V(t), \;\;\; t \ge 0.
\end{equation}
Let $\nabla X^x(t)$ be the derivative in probability   of the solution with respect to the initial value
$$
\left[ \nabla X^x(t)\right] _{i,j}= \frac{\partial }{\partial x_j} X^x_i(t).
$$
Note that, the process $X^x$ might not  be  integrable. However, as the noise is additive and $b$ has bounded derivative,  $ \nabla X^x(t)$ exists, it is $p$-integrable, for any $p\ge 1$, and
$$
\d \nabla X^x(t)= \nabla b(X^x(t))\nabla X^x(t)\d t, \qquad \nabla X^x(0)=I.
$$
Since $b$ has bounded derivatives, we have the next result in which $\|\cdot\|$ is any fixed norm on the space of real $d\times d$-matrices such that $\| CD\| $ $ \le \| C\| \| D\| $ for $d\times d$-matrices $C$ and $D$.
\begin{lemma}\label{L4}
For all $t\ge 0$ and $x\in \mathbb{R}^d$, $\nabla X^x(t)$ is an invertible matrix. Moreover, there is a constant $C$ such that
$$
\|\nabla X^x(t)\|+ \|\left(\nabla X^x(t)\right)^{-1}\|\le C\e^{Ct}, \qquad \forall\, t\ge 0, \ x\in \mathbb{R}^d.
$$
Moreover,  there is a constant $C$ such that
$$
 \|\nabla X^x(t)-I\| + \|\left(\nabla X^x(t)\right)^{-1}-I\|\le Ct.
$$
\end{lemma}
As a simple consequence of \REQ{E18} and Lemma \ref{L4} we have
$$
 \mathbb{D} X^x(t)= \nabla X^x(t)\int_0^t \left( \nabla X^x(s)\right)^{-1}\d Z^V(s).
$$
Let
\begin{equation}\label{E20}
M(t,x):= \int_0^t \left(\nabla X^x(s)\right)^{-1} \d Z^V(s).
\end{equation}
 Then  $\mathbb{D}X^x(t)= \nabla X^x(t)M(t,x)$ and consequently  the matrix valued process $A=[A_{k,j}(t,x)]$ given by \REQ{E10} satisfies
 \begin{equation}\label{E21}
 A(t,x)= \left(\mathbb{D}X^x(t)\right)^{-1} \nabla X^x(t) = \left(M(t,x)\right)^{-1}.
 \end{equation}
The proof of the following lemma is moved to the next section (Section \ref{S5}).
\begin{lemma}\label{L5}
Assume that  the parameter ${\delta}$ in \REQ{E12} is small enough.   Let $1\le p$. The Malliavin matrix  $\mathbb{D} X^x(t)$ is invertible and $p$-integrable. Moreover, the matrix valued process $A=[A_{k,j}(t,x)]$ given by \REQ{E10}  or \REQ{E21} is differentiable and $p$-integrable.
\end{lemma}
\section{Proof of Lemma \ref{L5}}\label{S5}
 Below, $\|\cdot\|$ denotes
 the operator 
 norm on the space of real $d\times d$-matrices. Moreover 
for a random $d \times d$-matrix  $B$ we set  
$$
 \| B \|_{L^p} = (\mathbb{E} \| B \|^p)^{1/p},\;\;\; p \ge 1.
$$
\begin{lemma}\label{L6}
(i) For any $t>0$, the matrix $Z^V(t)$ is invertible, $\mathbb{P}$-a.s.. Moreover, for any $p \ge 1$, $T>0$, there is a constant $C= C(p,T)$ such that
$$
\| \left(Z^V(t)\right)^{-1} \|_{L^p}\le  Ct^{-\frac{\kappa  }{\rho}}, \qquad t \in (0,T]. 
$$
(ii)  Assume that  the parameter ${\delta}$ in \REQ{E12} is  small enough (depending on the dimension $d$).   Then the  matrix $M(t,x)$ is invertible, $\mathbb{P}$-a.s.. Moreover, for any $p \ge 1$ and any  $T>0$, there is a constant $C= C_{p, T}$ such that
 \begin{equation} \label{wdd}
  \|A(t,x)- \left(Z^V(t)\right)^{-1}\|_{L^p}\le Ct^{-\frac{ \kappa  }{\rho} \, + \, 1}, \qquad t \in (0,T] 
\end{equation}  
 where $A(t,x)=\left(M(t,x)\right)^{-1}$.
\end{lemma}
\begin{proof} The first part of the lemma follows from Corollary \ref{C1} from Section \ref{S7} below. To show the second part note that
$$ 
M(t,x)= Z^V(t)+ \int_0^t R(s,x)\d Z^V(s),
$$
where  $R(t,x):= \left(\nabla X^x(t)\right)^{-1}-I$. Clearly, $R(t,x)$ is a random variable taking values in the space of $d\times d$ matrices. Note that 
 \begin{gather*}
\Big( \int_0^t R(s,x)\d Z^V(s) \Big)_{ij} = \int_0^t \int_{\mathbb{R}} R_{ij}(s,x) V_j ( s, \xi_j)  \Pi_j (ds, d \xi_j )
\\
= \Big(\sum_{0< s \le t} R(s,x) \tilde V(s, \triangle Z(s)) \Big)_{ij},
\end{gather*}
with $\triangle Z(s) =  Z(s) -  Z({s-})  $, where $\tilde V(s, z)$ is a diagonal matrix,   $s \ge 0,$ $z \in \R^d$, such that 
$$
(\tilde V(s, z))_{ii} = V_i(s,z_i),\;\;\; i =1, \ldots, d. 
$$
 Moreover,  $\mathbb{P}$-a.s., 
$Z^V(t)= \sum_{0< s \le t} \tilde V(s, \triangle Z(s))$  is convergent by \eqref{E4} and it is also invertible. 
 We write
\begin{gather} \label{inv1}
M(t,x) =   \Big (I+ \int_0^t R(s,x)\d Z^V(s)  \, (Z^V(t))^{-1}\Big) Z^V(t).
\end{gather}
We would like to obtain, for $\delta >0$ small enough, $t>0$, 
\begin{gather} \label{d11}
 A(t,x) = (Z^V(t))^{-1} \, \Big (I+ \int_0^t R(s,x)\d Z^V(s)  \, (Z^V(t))^{-1}\Big)^{-1}.
\end{gather}
To this purpose we consider
$$
Q(t,x)= \int_0^t R(s,x)\d Z^V(s)  \, (Z^V(t))^{-1}.
$$
Recall that  $(e_j)$ is  the canonical basis of  $\R^d$.  We get for $j =1, \ldots, d$, $\P$-a.s., 
\begin{align*}
 Q(t,x) e_j 
&= \sum_{0< s \le t} R(s,x) \tilde V(s, \triangle Z(s)) e_j \, 
\big(  \int_0^t \int_{\mathbb{R}}  V_j ( r, \xi_j)  \Pi_j (dr, d \xi_j )\big)^{-1}
\\
&= \sum_{0< s \le t} R(s,x)  V_j(s, \triangle Z_j(s) ) e_j \, 
\Big(  \int_0^t \int_{\mathbb{R}}  V_j ( r, \xi_j)  \Pi_j (dr, d \xi_j )\Big)^{-1}
\end{align*}
and so
\begin{align*}
| Q(t,x) e_j| &\le  \sum_{0< s \le t} \| R(s,x) \| \,  V_j(s, \triangle Z_j(s) )  \, 
\Big(  \int_0^t \int_{\mathbb{R}}  V_j ( r, \xi_j)  \Pi_j (dr, d \xi_j )\Big)^{-1}
\\
&\le c  t   \int_0^t \int_{\mathbb{R}}  V_j ( r, \xi_j)  \Pi_j (dr, d \xi_j ) \, \Big(  \int_0^t \int_{\mathbb{R}}  V_j ( s, \xi_j)  \Pi_j (ds, d \xi_j )\Big)^{-1} \\
&\le \min (ct,1/2),
\end{align*}
where $c$ is independent of $x \in \R^d$ and $\omega$,  $\P$-a.s.
Therefore, as $Z^V(t)$ is invertible, the matrix $M(t,x)$ is invertible and $A(t,x)=\left(M(t,x)\right)^{-1}$ satisfies \eqref{d11}. Moreover
\begin{align} \label{cinque}
A(t,x)  =   \left(Z^V(t)\right)^{-1} + \left(Z^V(t)\right)^{-1}\sum_{n=1}^{+\infty} (-1)^n (Q(t,x))^n.   
\end{align}
Consequently, we have 
$$
\left \| A(t,x)-\left(Z^V(t)\right)^{-1}\right\|_{L^p}\le  \min\{ C_2t, 1/2\} \, \| \left(Z^V(t)\right)^{-1} \|_{L^p}
$$
and \eqref{wdd} follows. The proof is complete.
\end{proof}
 \begin{remark} \label{comm} {\rm 
We note that in the previous proof  it is  important  to have  a term like  
$
\int_0^t R(s,x)\d Z^V(s)  \, (Z^V(t))^{-1}
$ (cf. \eqref{inv1}). Such term can be estimated in a sharp way by 
$\min (ct, 1/2)$. 
On the other hand, a term like $ (Z^V(t))^{-1}\int_0^t R(s,x)\d Z^V(s)  \, $ would be  difficult to estimate in a sharp way (we can estimate the $L^2$-norm by  $C t^{-\frac{ \kappa  }{\rho} \, + \, 3/2}$).  On this respect see  also the computations in  Section 6.2.
}  
\end{remark}  

\subsection{Proof of Lemma \ref{L5}}\label{S61}
Since $b$ has bounded derivatives of the first and second order, $\nabla X^x(t)$  and $\left(\nabla X^x(t)\right)^{-1}$ are  differentiable and $p$-integrable.  Next,
thanks to \REQ{E6}, the matrix valued process $Z^V$ given by \REQ{E19} is also differentiable, $p$-integrable, and
\begin{equation}\label{E22}
\begin{aligned}
D_k Z^V_{kk}(t)&= \frac{\d }{\d \varepsilon} \int_0^t \int_{\mathbb{R}}V_k(s, \xi_k +\varepsilon V_k(s,\xi_k))\Pi_k(\d s, \d \xi_k)\vert_{\varepsilon =0} \\
&= \int_0^t \int_{\mathbb{R}}\psi^2_{{\delta}}(s)\phi_\delta(\xi_k) \phi_\delta' (\xi_k)\Pi_k(\d s, \d \xi_k).
\end{aligned}
\end{equation}
Therefore, $\mathbb{D}X^x(t)$ is differentiable and $p$-integrable.  Clearly $\nabla X^x(t)$ is invertible.  By Lemma \ref{L6}, the matrix  $M(t,x)$ given by \REQ{E20} is invertible and clearly differentiable and  $p$-integrable. We can show the  differentiability of $ \left( \mathbb{D}X^x(t)\right)^{-1}$ or equivalently of $A(t,x)$ in a standard way based on the observations that 
$$
D_k \left( \mathbb{D}X^x(t)\right)^{-1}= - \left( \mathbb{D} X^x(t)\right)^{-1}\left( D_k \mathbb{D} X^x(t)\right)  \left( \mathbb{D} X^x(t)\right)^{-1}. \qquad \square
$$

\section{Proof of Theorem \ref{T1}}\label{S6}
By Lemma \ref{L5} the random field  $Y(t,x)$ given by \REQ{E9} is well defined and integrable.  Clearly,   by the standard approximation argument it is enough to show that for any $f\in C_b^1(\mathbb{R}^d)$ we have  \REQ{E3}. To this end note that
\begin{align*}
\nabla P_t f(x)&= \nabla \mathbb{E}\, f(X^x(t))= \mathbb{E}\, \nabla f(X^x(t))\nabla X^{x}(t).
\end{align*}
Since, by Lemma \ref{L1},
$$
\mathbb{D} f(X^x(t))= \nabla f(X^x(t))\mathbb{D} X^x(t),
$$
and, by Lemma \ref{L5}  the matrix $\mathbb{D} X^x(t)$ is invertible, we have
\begin{align*}
\nabla P_tf(x)&= \mathbb{E}\left( \mathbb{D} f(X^x(t))\right) \left[ \mathbb{D} X^x(t)\right] ^{-1}\nabla X^x(t)\\
&= \mathbb{E}\left(  \mathbb{D} f(X^x(t)) \right)A(t,x)
=\sum_{j=1}^d \sum_{k=1}^d \mathbb{E}\, D_kf(X^x(t))A_{k,j}(t,x) { e_j},
\end{align*}
where $A(t,x)$ is given by \REQ{E10} or equivalently by \REQ{E21}. Since
\begin{align*}
\sum_{k=1}^d D_kf(X^x(t))A_{k,j}(t,x)&= \sum_{k=1}^d \left\{ D_k\left[ f(X^x(t))A_{k,j}(t,x)\right] - f(X^x(t))D_kA_{k,j}(t,x)\right\},
\end{align*}
we have \REQ{E3} with $Y$ given by \REQ{E9}. Cleary, the same arguments can be apply to show the BEL formula for the L\'evy semigroup.

The proof of \REQ{E7} and \REQ{E8}   is more difficult, and it is divided into the following two parts.
\subsection{L\'evy case} Assume that $b\equiv 0$, that is $X^x(t)=  Z^x(t)$. Let us fix a time horizon $T<+\infty$. We are proving estimate \REQ{E7} for the process $Y(t)$ corresponding to the pure L\'evy case.

We have
$$
Y_j(t)= \sum_{k=1}^d \left[ A_{k,j}(t)D^*_k\mathbf{1}(t)- D_k A_{k,j}(t)\right],
$$
where $A(t)= \left[ \mathbb{D}Z^x(t)\right]^{-1}= \left[ Z^V(t)\right]^{-1}$ and $Z^V(t)$ is a diagonal matrix defined in \REQ{E19}. Therefore
$$
Y_{j}(t)=\frac{D_j^*\mathbf{1}(t)}{Z_{jj}^V(t)} - D_j  \frac{1}{Z_{jj}^V(t)} =  \frac{D_j^*\mathbf{1}(t)}{Z_{jj}^V(t)} +  \frac{D_j Z^{V}_{jj}(t)}{\left(Z_{jj}^V(t)\right)^2},
$$
where  $D_j^*\mathbf{1}(t)$ and  $D_j Z^{V}_{jj}$ are given by \REQ{E17} and \REQ{E22}, respectively.  We have
\begin{align*}
\mathbb{E}\left\vert  D_j^*\mathbf{1}(t) \left(Z_{jj}^V(t)\right)^{-1} \right\vert &\le \left(\mathbb{E}\left\vert  D_j^*\mathbf{1}(t)\right\vert ^2\right)^{1/2} \left(\mathbb{E}\left\vert  Z_{jj}^V(t) \right\vert ^{-2}\right)^{1/2}.
\end{align*}
By Lemma  \ref{L6}, there is a constant $C_1$ such that $\mathbb{E}\left\vert  Z_{jj}^V(t) \right\vert ^{-2}\le C_1t^{-2\kappa /\rho }$.  Next there are constants $C_2$ and $C_3$ such that
\begin{equation}\label{Ed11}
\mathbb{E}\left\vert D_j^*\mathbf{1}(t)\right\vert ^2 \le C_2\int_0^t \int_{-\delta}^{\delta}  \left\vert \frac{\frac {\d }{\d  \xi} \left(V_j(s,\xi_j)\rho_j(\xi _j)\right)} {\rho_j(\xi _j)}\right\vert ^2 \rho_j(\xi_j)\d \xi_j \d s\le C_3 t,
\end{equation}
where the last estimate follows from \REQ{E5}. Therefore there is a constant $C_4$ such that
\begin{equation}\label{E23}
\mathbb{E}\left\vert  D_j^*\mathbf{1}(t) \left(Z_{jj}^V(t)\right)^{-1} \right\vert\le C_4 t^{- \frac{\kappa}{\rho} +\frac{1}{2}}, \qquad t\in (0,T].
\end{equation}
Let us observe now that
\begin{align*}
&\left \vert D_j Z^{V}_{jj}(t)\right\vert =  \left\vert \int_0^t \int_{\mathbb{R}}\psi^2_{{\delta}}(s)\phi_\delta(\xi_j) \phi_\delta' (\xi_j)\Pi_j(\d s, \d \xi_j)\right\vert \\
&\le  \left( \int_0^t \int_{\mathbb{R}}\psi^2_{{\delta}}(s)\phi_\delta^2(\xi_j) \Pi_j(\d s, \d \xi_j)\right)^{1/2}
\left( \int_0^t \int_{\mathbb{R}}\psi^2_{{\delta}}(s) \left( \phi_\delta' (\xi_j)\right)^2 \Pi_j(\d s, \d \xi_j)\right)^{1/2}\\
&\le   \int_0^t \int_{\mathbb{R}}\psi_{{\delta}}(s)\phi_\delta(\xi_j) \Pi_j(\d s, \d \xi_j)
\left( \int_0^t \int_{\mathbb{R}}\psi^2_{{\delta}}(s) \left( \phi_\delta' (\xi_j)\right)^2 \Pi_j(\d s, \d \xi_j)\right)^{1/2};
\end{align*}
here in the last inequality we have used an elementary relation
$$
\sum_{k}x_k^2\leq \left(\sum_{k}x_k\right)^2,
$$
valid for any non-negative real numbers $\{x_k\}$.
 Thus
\begin{equation}\label{E24}
\left \vert D_j Z^{V}_{jj}(t)\right\vert \le Z^V_{jj}(t) \left( \int_0^t \int_{\mathbb{R}} \psi^2_{{\delta}}(s) \left( \phi_\delta' (\xi_j)\right)^2 \Pi_j(\d s, \d \xi_j)\right)^{1/2}.
\end{equation}
Therefore, by Lemma  \ref{L6},
\begin{align} \label{svv}
\nonumber \mathbb{E} \left\vert  \frac{D_j Z^{V}_{jj}{ (t)}}{\left(Z_{jj}^V(t)\right)^2}\right\vert &\le \mathbb{E}\, \frac{\left( \int_0^t \int_{\mathbb{R}}\psi^2_{{\delta}}(s)  \left( \phi_\delta'(\xi_j)\right)^2 \Pi_j(\d s, \d \xi_j)\right)^{1/2}}{ \int_0^t \int_{\mathbb{R}}\psi_{{\delta}}(s)\phi_\delta(\xi_j) \Pi_j(\d s, \d \xi_j)}\\ \nonumber
&\le\left(  \mathbb{E}\,  \int_0^t \int_{\mathbb{R}} \psi^2_{{\delta}}(s) \left( \phi_\delta'(\xi_j)\right)^2 \Pi_j(\d s, \d \xi_j)\right)^{1/2}\\
&\qquad \times \left( \mathbb{E}\left(
\int_0^t \int_{\mathbb{R}}\psi_{{\delta}}(s)\phi_\delta (\xi_j) \Pi_j(\d s, \d \xi_j)\right)^{-2}\right)^{1/2} \le C_5 t^{ - \frac{\kappa}{\rho}+\frac{1}{2}}.
\end{align}
Note that  $\int_{\mathbb{R}} \left( \phi_\delta'(\xi_j)\right)^2 m_j(\d\xi_j)<+\infty$ thanks to \REQ{E6}.  Summing up, we can find a constant $C$ such that
\begin{equation} \label{wr}
  \mathbb{E}\left\vert Y(t)\right\vert \le Ct^{- \frac{\kappa}{\rho}+\frac{1}{2} }
 \end{equation}
which is the desired estimate. $\square$
\subsection{General  case}
Recall that  $M$ and $A=M^{-1}$ are  given by \REQ{E20} and \REQ{E21}, respectively. Let $T>0.$ We  prove  first that (for $\delta >0$ small enough) there  is a constant $c$ such that
\begin{gather}\label{E25}
\mathbb{E}\left\vert D_k^*\mathbf{1}(t)A_{j,k}(t,x)\right\vert \le c t ^{-\frac{\kappa}{\rho}+\frac{1}{2}}, 
\\
\label{Ewe}
\mathbb{E}\left\vert D_k^*\mathbf{1}(t)A_{j,k}(t,x)- D_k^*\mathbf{1}(t)  \left(Z_{jj}^V(t)\right)^{-1} \right\vert \le 
 c t^{-\frac{ \kappa  }{\rho} \, + \, 3/2 },\;\; \;\; t\in (0,T].   
\end{gather}
By Lemma \ref{L6} there is a   constant $C>0$ such that
$$
 \|A(t,x)- \left(Z^V(t)\right)^{-1}\|_{L^q}\le C t^{-\frac{ \kappa  }{\rho} \, + \, 1}, \;\; t\in (0, T], \;\; q \ge 1.
$$
Therefore, \eqref{Ewe} follows from 
\eqref{Ed11} by using the Cauchy-Schwarz  inequality.  
Clearly \eqref{E25} follows from 
\eqref{E23} 
and  
   \eqref{Ewe}.

It is much harder to evaluate $L^1$-norm of the term
\begin{equation}\label{E26}
I(t,x):= \sum_{j=1}^d \sum_{k=1}^d D_k A_{k,j}(t,x)e_j = \sum_{j=1}^d \sum_{k=1}^d \left[ A(t,x) (D_k M(t,x)) A(t,x)\right]_{k,j}e_j.
\end{equation}
Recall that   $R(s,x):= \left(\nabla X^x(s)\right)^{-1}-I$;  moreover, 
$$
M(t,x)= Z^V(t)+ \int_0^t R(s,x)\d Z^V(s)
$$
is differentiable, $p$-integrable,  and we have (see also \eqref{E22}):
 \begin{equation} \label{ci1}
   D_kM(t,x)=D_kZ^V(t)+  \int_{0}^t R(s,x)\d D_k Z^V(s) + \int_{0}^t D_kR(s,x)\d Z^V(s).
\end{equation}
We have $\|R(s,x)\|\le C_1s$ and  it is not difficult to see that there is a random variable $\eta>0$ integrable with arbitrary power,
such that 
\begin{equation}\label{E29}
\|D_kR(s,x)\|\le \eta s^2 , \qquad s\in (0,T].
\end{equation}
We will  show that $I(t,x)$ is a proper perturbation of  the already estimated
\begin{equation}\label{E27}
I_0(t):= \sum_{j=1}^d \frac{D_j Z^{V}_{jj}(t)}{\left(Z_{jj}^V(t)\right)^2}e_j = \sum_{j=1}^d \sum_{k=1}^d \left[  \left(Z^V(t)\right)^{-1}(D_k Z^V(t)) \left(Z^V(t)\right)^{-1} \right]_{k,j}e_j.
\end{equation}
The proof will be completed as soon as we can show   there is a constant $C_1$ such that  
\begin{equation}\label{E281}
\mathbb{E}\left\vert  I(t,x)-I_0(t)\right\vert  \le 
C_{1}    t^{-\frac{ \kappa  }{\rho} +3/2 },\;\;\; t \in (0,T]. 
\end{equation}
This will imply that  
\begin{equation}\label{w44}
\E |I(t,x)| \le \mathbb{E}\left\vert  I(t,x)-I_0(t)\right\vert  + 
 \mathbb{E}\left\vert  I_0(t)\right\vert \le \tilde C t^{- \frac{\kappa}{\rho}+\frac{1}{2} }. 
\end{equation}
Collecting  \eqref{E25} and \eqref{w44} will give the estimate for $\E |Y(t,x)|$.
 
Let us prove \eqref{E281}.  Recalling that $A(t,x) = (M(t,x))^{-1}$ we have to find an  $L^1$-bound for 
$$
 \|  A(t,x) (D_k M(t,x)) A(t,x) -  (Z^V(t))^{-1}(D_k Z^V(t)) \left(Z^V(t)\right)^{-1}  \|  \le  J_1 + J_2 + J_3,
$$
\begin{gather*}
 J_1 =  \|  A(t,x) (D_k M(t,x)) [A(t,x) -   \left(Z^V(t)\right)^{-1}]  \|,
 \\
 J_2 = \| A(t,x) (D_k M(t,x) - D_k Z^V(t)) \, (Z^V(t))^{-1} \|,
 \\
  J_3 = \| [A(t,x) - (Z^V(t))^{-1}] \,  D_k Z^V(t)) \, (Z^V(t))^{-1} \|.
\end{gather*}
As for $J_3$, using \eqref{wdd} and \eqref{E24} we infer
\begin{gather*}
\E [J_3] \le  \| [A(t,x) - (Z^V(t))^{-1}] \| _{L^2} \,  \| D_k Z^V(t)) \, (Z^V(t))^{-1} \|_{L^2}
\end{gather*}
Using \eqref{wdd} we infer
\begin{gather*}
\E [J_3] \le C t^{-\frac{ \kappa  }{\rho} \, + \, 1  }  \| D_k Z^V(t)) \, (Z^V(t))^{-1} \|_{L^2}.
\end{gather*}
Since $ \| D_k Z^V(t)) \, (Z^V(t))^{-1}\|  = \Big |\frac{D_k Z^{V}_{kk}{ (t)}}{Z_{kk}^V(t)} \Big |$, we can use 
 \eqref{E24} and get 
$$ 
 \E \| D_k Z^V(t)) \, (Z^V(t))^{-1}\|^2
\le 
 \E \left( \int_0^t \int_{\mathbb{R}} \psi^2_{{\delta}}(s) \left( \phi_\delta' (\xi_j)\right)^2 \Pi_j(\d s, \d \xi_j)\right) \le C_4 t
$$
(see \eqref{E6}). We obtain
\begin{gather*}
\E [J_3] \le C t^{-\frac{ \kappa  }{\rho} \, + \, 3/2  }.
\end{gather*}
Concerning $J_2$ we find by \eqref{ww}
\begin{gather*}
 \E [J_2] \le  \| A(t,x) \|_{L^2 } \, \| (D_k M(t,x) - D_k Z^V(t)) \, (Z^V(t))^{-1} \|_{L^2} 
 \\ \le  C t^{-\frac{ \kappa  }{\rho}   } \,  \| (D_k M(t,x) - D_k Z^V(t)) \, (Z^V(t))^{-1} \|_{L^2}. 
\end{gather*}
Now
\begin{equation}\label{rit}
\begin{aligned} 
  &(D_k M(t,x) - D_k Z^V(t)) \, (Z^V(t))^{-1}
  \\ \qquad &= \Big(\int_{0}^t R(s,x)\d D_k Z^V(s) + \int_{0}^t D_kR(s,x)\d Z^V(s) \Big) \,  (Z^V(t))^{-1}.
\end{aligned}
\end{equation}
We will argue as in the proof of Lemma \ref{L6}. Recall that, for $\delta  $ small enough, 
 \begin{align*}
\Big (\int_{0}^t R(s,x)\d D_k Z^V(s) \Big)_{ij} &=   \int_0^t \psi^2_{{\delta}}(s) R_{ij}(s,x) \int_{\mathbb{R}}   \phi_\delta(\xi_j) \phi_\delta' (\xi_j)
\Pi_j (ds, d \xi_j )
\\ 
&= \Big(\sum_{0< s \le t} R(s,x) \tilde U(s, \triangle Z(s)) \Big)_{ij},
\end{align*} 
where   $\tilde U(s, z)$ is a diagonal matrix,   $s \ge 0,$ $z \in \R^d$, such that 
$$
(\tilde U(s, z))_{ii} = U_i(s, z) =  \psi^2_{{\delta}}(s) \phi_\delta(z_i) \phi_\delta' (z_i),
\;\;\;\; (s,z_i),\;\;\; i =1, \ldots, d. 
$$
Hence
\begin{align*}
 &\Big | \int_{0}^t R(s,x)\d D_k Z^V(s) \,  (Z^V(t))^{-1} e_j \Big|
 \\
 &\quad = \Big | \sum_{0< s \le t} R(s,x) \tilde U(s, \triangle Z(s)) e_j \, 
\big(  \int_0^t \int_{\mathbb{R}}  V_j ( r, \xi_j)  \Pi_j (dr, d \xi_j )\big)^{-1} \Big |
\\
&\quad \le  \sum_{0< s \le t}  \| R(s,x) \|  U_j(s, \triangle Z_j(s) ) \, 
\Big(  \int_0^t \int_{\mathbb{R}}  V_j ( r, \xi_j)  \Pi_j (dr, d \xi_j )\Big)^{-1} 
\\ &\quad \le ct  \sum_{0< s \le t}   U_j(s, \triangle Z_j(s) ) \, 
\Big(  \int_0^t \int_{\mathbb{R}}  V_j ( r, \xi_j)  \Pi_j (dr, d \xi_j )\Big)^{-1}   = ct\,\frac{D_j Z^{V}_{jj}(t)}{ Z_{jj}^V(t)},  
\end{align*}
see  \eqref{E22}. We deduce that 
$$
\Big \| \int_{0}^t R(s,x)\d D_k Z^V(s) \,  (Z^V(t))^{-1} \Big \|_{L^2} \le c t^{3/2}.
$$
Now in order to estimate $J_2$ it remains to consider 
$$
\int_{0}^t D_kR(s,x)\d Z^V(s)  \,  (Z^V(t))^{-1} = 
 \sum_{0< s \le t} D_k R(s,x) \tilde V(s, \triangle Z(s)) (Z^V(t))^{-1},
$$
   where $\tilde V(s, z)$ is a diagonal matrix,   $s \ge 0,$ $z \in \R^d$, such that  $
(\tilde V(s, z))_{ii} $ $ = V_i(s,z_i).$  Using also the bound \eqref{E29} we obtain, for $j =1, \ldots, d,$ 
 \begin{align*}
&\Big |\int_0^t D_k R(s,x)\d Z^V(s) \,  \left(Z^V(t)\right)^{-1} e_j \Big| 
 \\
 &\quad \le  \sum_{0< s \le t} \| D_k R(s,x) \| \,  V_j(s, \triangle Z_j(s) )  \, 
\Big(  \int_0^t \int_{\mathbb{R}}  V_j ( r, \xi_j)  \Pi_j (dr, d \xi_j )\Big)^{-1} \\
&\quad \le \eta t^2.
\end{align*}
It follows that 
$$ 
\Big \|\int_0^t D_k R(s,x)\d Z^V(s) \,  \left(Z^V(t)\right)^{-1} e_j \Big \|_{L^2}  \le C t^2. 
$$
We finally obtain
\begin{gather*}
 \E [J_2] \le C t^{-\frac{ \kappa  }{\rho} \, + \, 3/2  }.
\end{gather*}
To treat $J_1$ we note that 
\begin{equation}\label{s33}
\|  A(t,x)\|_{L^2} \le \|  A(t,x) - \left(Z^V(t)\right)^{-1}\|_{L^2}
+ \|  \left(Z^V(t)\right)^{-1}\|_{L^2} \le C  t^{-\frac{ \kappa  }{\rho}},   \;\;\; t \in (0,T],
\end{equation}
see Lemma \ref{L6}. Hence
\begin{equation} \label{s2r}
\begin{aligned}
 &\E [J_1] \le  \|  A(t,x)\|_{L^2}  \| (D_k M(t,x)) [ A(t,x) -   \left(Z^V(t)\right)^{-1}]  \|_{L^2}
\\ &\qquad  \le  C    t^{-\frac{ \kappa  }{\rho}   } \, \| (D_k M(t,x)) [ A(t,x) -   \left(Z^V(t)\right)^{-1} ]   \|_{L^2 }.
\end{aligned}
\end{equation}
 We write
 \begin{align*}
&\| (D_k M(t,x)) [ A(t,x) -   \left(Z^V(t)\right)^{-1}]\| 
\\&\quad  = \|(D_k M(t,x)) (M(t,x))^{-1}  M(t,x) [ A(t,x) -   \left(Z^V(t)\right)^{-1}]\|
\\
&\quad \le   \|(D_k M(t,x)) (M(t,x))^{-1} \| \, \| I - M(t,x)\left(Z^V(t)\right)^{-1} \|.
\end{align*}
The more difficult term is 
\begin{align*}
  \|(D_k M(t,x)) (M(t,x))^{-1} \| &=  \|(D_k M(t,x)) \left(Z^V(t)\right)^{-1} \, \left(Z^V(t)\right) (M(t,x))^{-1} \|
  \\
  &\le \|(D_k M(t,x)) \left(Z^V(t)\right)^{-1} \| \, \| \left(Z^V(t)\right) (M(t,x))^{-1} \|.
\end{align*}
Now, see \eqref{cinque}, 
\begin{align*}
&\left(Z^V(t)\right) (M(t,x))^{-1} = \left(Z^V(t)\right) A (t,x)
\\ 
&\quad = \left(Z^V(t)\right) \Big (
\left(Z^V(t)\right)^{-1} + \left(Z^V(t)\right)^{-1}\sum_{n=1}^{+\infty} (-1)^n (Q(t,x))^n \Big )\\
&\quad =     \sum_{n=0}^{+\infty} (-1)^n (Q(t,x))^n.
\end{align*}
Hence, $\P$-a.s., 
\begin{gather*}
 \| \left(Z^V(t)\right) (M(t,x))^{-1} \| \le C_1
\end{gather*}
where $C_1 $  is independent of $x$, $t \in (0,T]$ and $\omega$, $\P$-a.s..
 The term  
 $$
 \|(D_k M(t,x)) \left(Z^V(t)\right)^{-1} \|
 $$
 can be treated as  the one in \eqref{rit}. We obtain
\begin{gather*}
\|(D_k M(t,x)) \left(Z^V(t)\right)^{-1} \|_{L^2} \le c t^{1/2}.
\end{gather*}
We infer that 
\begin{equation}\label{w55}
 \|(D_k M(t,x)) (M(t,x))^{-1} \|_{L^2} \le c_1 t^{1/2},\;\; t \in (0,T].  
\end{equation}
It remains to consider 
$$
\| I - M(t,x)\left(Z^V(t)\right)^{-1} \| = \Big \|  \int_0^t R(s,x)\d Z^V(s) \,  \left(Z^V(t)\right)^{-1}  \Big \| \le c_3t
$$
where $c_3 $ is independent of $x$ and $\omega$, $\P$-a.s.. Finally we have
\begin{equation}\label{j11}
\E [J_1] \le C_0  t^{-\frac{ \kappa  }{\rho}  + 3/2},
\end{equation}
and the proof is complete.
$\square$

\section{An integrability result}\label{S7} 
 
 
Assume that $\mathcal{M}$ is a Poisson random measure on $[0,+\infty)\times \mathbb{R}$ with intensity measure $\d t m(\d \xi)$. Given  a measurable $h\colon \mathbb{R}\mapsto [0,+\infty)$ let
$$
J_h(t):= \int_0^t \int_{\mathbb{R}} h(\xi)  \mathcal{M}(\d s, \d \xi).
$$
Then
$$
\mathbb{E} \e^{-\beta J_h(t)} = \exp \left\{ -t \int_{\mathbb{R}} \left(1- \e^{-\beta h(\xi)}\right) m(\d \xi)\right\}.
$$
Using the identity
$$
y^{-q} = \frac{1}{\Gamma(q)} \int_0^{+\infty} \beta ^{q-1} \e^{-\beta y} \d \beta,\qquad y>0,
$$
we obtain
\begin{align*}
\mathbb{E}\, J_h(t)^{-q} &= \frac{1}{\Gamma(q)} \int_0^{+\infty} \beta ^{q-1} \mathbb{E}\, \e^{-\beta J_h(t)} \d \beta \\
&= \frac{1}{\Gamma(q)} \int_0^{+\infty} \beta ^{q-1}  \exp\left\{ -t \int_{\mathbb{R}} \left(1- \e ^{-\beta h(\xi )}\right) m(\d \xi )\right\} \d \beta.
\end{align*}
Using this method one can immediately obtain (see Norris \cite{Norris}) the following result.
\begin{lemma}\label{L7} If for a certain $\rho >0$,
$$
 \liminf_{\varepsilon \downarrow 0} \varepsilon ^\rho m\{h\ge \varepsilon\}>0,
$$
then
$$
\mathbb{E}\, J_h(t)^{-q} \le Ct^{-\frac{q}{\rho}}, \qquad q\ge 1,\ t\in (0,1].
$$
\end{lemma}

Let $\phi_\delta\in C^\infty(\mathbb{R}\setminus\{0\})$ be given by \REQ{E7}.  Applying Lemma \ref{L7} to $m(\d \xi)$ satisfying hypothesis $(ii)$ of Theorem \ref{T1}  and $h=\phi_\delta$ we obtain the following result.
\begin{corollary}\label{C1}
For any $q\ge 1$ there is a constant $C= C(q,T)$ such that
$$
\mathbb{E}\, J_{\phi_\delta}(t)^{-q}\le Ct^{-\frac{\kappa q}{\rho}}, \qquad t\in (0,T].
$$
Moreover,
$$
\mathbb{E} \, J_{\phi_\delta}(t)= t \int_{\mathbb{R}}\phi_\delta (\xi ) m(\d \xi)<+\infty.
$$
\end{corollary}

\section
{Sharp estimates in the cylindrical $\alpha$-stable case }\label{vv}


\vskip 0.1 cm  
 Here we are concerned with rather general    perturbation of  $\alpha$-stable case.   Indeed in such  case we can improve the estimate on $Y(t)$ given  in Section 6.1. This estimate according to Remark \ref{qq} leads to the sharp gradient estimates \eqref{w22}.

Below in \eqref{E51} we will strengthen hypotheses \eqref{E5} and \eqref{E6}. In Remark  \ref{s11} we clarify the validity of  the  new assumptions in  the relevant  cylindrical $\alpha$-stable case.
  \begin{lemma} \label{cil}
Let $\alpha \in (0,2)$.   Suppose that  all the assumptions
of Theorem \ref{T1} hold with $\rho = \alpha$ and for any $\kappa > 1 + \alpha /2$.
Moreover, suppose that
\begin{align}\label{E51}
 \limsup_{r \to 0^+}\;  r^{\frac{ - 2\kappa + 2}{\alpha}  + 1 }  \int_{-r} ^{r} \Big [ |\xi|^{2\kappa} \left( \frac{\rho'_j(\xi)}{\rho_j(\xi)}\right)^2 +   |\xi|^{2\kappa -2} \Big] \rho_j(\xi)\d \xi&<+\infty.
\end{align}
Then the following estimate holds for the $\mathbb{R}^d$-valued  process $Y$ (cf \eqref{wr}):
\begin{equation}\label{see}
\mathbb{E} \left\vert Y(t)\right\vert  \le C t^{-\frac{1}{\alpha}}, \qquad t \in (0,T].
\end{equation}
\end{lemma}
\begin{remark}\label{s11} {\rm We provide a sufficient condition such that
 all the hypotheses  of Lemma  \ref{cil} hold.
 To this purpose  recall that  $\rho_j$ is the $C^1$-density of the L\`evy measure $m_j$ associated to the process $Z_j$;   such density exists   on $(-\delta, \delta) \setminus \{0\}$, $\delta >0$.

 Moreover, $l_\alpha (\xi) := \vert \xi \vert ^{-1-\alpha }$ denotes  the density of the L\'evy measure of a symmetric one-dimensional $\alpha$-stable process, $\alpha \in (0,2)$.

 Assume that there is a positive constant $c$ such that, for $\xi \in (-\delta, \delta) \setminus \{ 0\},$
\begin{equation} \label{wq}
\left\vert \frac{\rho'_j(\xi)}{\rho_j (\xi)}\right\vert \le c\left( \left\vert \xi\right\vert ^{-1} +1\right)\qquad \text{and}\qquad c^{-1}l_\alpha(\xi)\le \rho_j(\xi)\le c l_\alpha(\xi),
\end{equation}
 $j =1, \ldots, d$. It is easy to check that \eqref{wq} implies all the assumptions of Lemma \ref{cil}. Thus under condition \eqref{wq} we obtain \eqref{see} and the sharp gradient estimates \eqref{w22}.
}
\end{remark}

\begin{proof}
To prove the result we can assume $d=1$ so that  $Y_1 = Y $; $\Pi $ is the associated  Poisson random measure  and we set  $m_1 =\mu$ for the corresponding L\'evy measure having $C^1$-density $\rho$ on $(-\delta, \delta)$.

It is enough to show \eqref{see} for small $t$,  say  $t^{1/\alpha} \wedge t \leq \delta/2 \le 1$.   Note that, for {  $|\xi|\le \delta/2$, $\phi_\delta (\xi)=|\xi|^{\kappa}$.}

Let us fix
$
 \kappa =  1+ \frac{3}{4} \alpha.
$
 We have
$$
Y(t)=\frac{D^*\mathbf{1}(t)}{Z^V(t)} - D \frac{1}{Z^V(t)} =  \frac{D^*\mathbf{1}(t)}{Z^V(t)} +  \frac{DZ^{V}}{\left(Z^V(t)\right)^2}.
$$
We have
\begin{align*}
D^*\mathbf{1}(t)&= -  \int_0^t \int_{(-\delta,\delta)}  \frac{\phi_\delta'(\xi) \rho(\xi )+ \phi_\delta(\xi)\rho'(\xi)} {\rho(\xi)} { \hat \Pi ( \d s, \d \xi),}\\
 D Z^{V}(t) &=   \int_0^t \int_{(-\delta, \delta)}\phi_{\delta}(\xi)  \phi_\delta'(\xi)\Pi (\d s, \d \xi),\\
 Z^V(t) &= \int_0^t \int_{(-\delta,\delta)} \phi_{\delta}(\xi) \Pi(\d s,\d \xi).
 \end{align*}
 We are showing that
 \begin{equation}\label{sdd1}
 \mathbb{E} \left\vert  \frac{ D^*\mathbf{1}(t)}{Z^V(t)}\right\vert \le C _2 t^{- \frac{1}{\alpha}},
 \end{equation}
We concentrate on $D^*\mathbf{1}(t) $:
\begin{gather*}
D^*\mathbf{1}(t) = I_1 (t) + I_2(t), \;\;\;
I_1(t) =    \int_0^t \int_{\{ t^{1/\alpha} < |\xi| <\delta\} }  \frac{\phi_\delta'(\xi) \rho(\xi )+ \phi_\delta(\xi)\rho'(\xi)} {\rho(\xi)}
{ \hat \Pi ( \d s, \d \xi),}
\\ I_2(t)
 =   \int_0^t \int_{\{|\xi| \le t^{1/\alpha} \} }  \frac{\phi_\delta'(\xi) \rho(\xi )+ \phi_\delta(\xi)\rho'(\xi)} {\rho(\xi)}
{ \hat \Pi ( \d s, \d \xi).}
\end{gather*}
  Concerning $I_1(t)$ we can improve  some estimates  of Section 6.1; using the H\"older inequality (because $\xi$ is separated from $0$):
  for  $q>2,$ $p \in (1,2) \colon 1/p+1/q=1$ we have
\begin{align*}
\mathbb{E}\left\vert   I_1(t) \left(Z_{}^V(t)\right)^{-1} \right\vert &\le \left(\mathbb{E}\left\vert
 I_1(t)
\right\vert ^p\right)^{1/p} \left(\mathbb{E}\left\vert  Z_{}^V(t) \right\vert ^{-q}\right)^{1/q}.
\end{align*}
By Corollary  \ref{C1}, there is a constant $C_1$ such that $\mathbb{E}\left\vert  Z_{}^V(t) \right\vert ^{-q}\le C_1t^{-\kappa q/\rho }$.  Since $p\in (1,2)$,    there are constants $C_2$ and $C_3$ such that
$$
\mathbb{E}
| I_1(t)
 |^p \le C_2\int_0^t \int_{ \{ t^{1/\alpha} < |\xi| <\delta\}  }
 \left\vert 
  \frac{\phi_\delta'(\xi) \rho(\xi )+ \phi_\delta (\xi)\rho'(\xi)} {\rho(\xi)}
\right\vert ^p \rho(\xi)\d \xi \d s\le C_3 t,
$$
(see { e.g. Lemma 8.22} in \cite{Peszat-Zabczyk}) where the last estimate follows from \REQ{E5}.
  Choosing  $\varepsilon \in (0,1/4) $  so that
$
-\frac{\kappa}{\alpha} $ $ +1-\varepsilon$ $ >-\frac{1}{\alpha},
$
and taking $q$ large, we can see that  there exists  a constant $c_\varepsilon$ such that
\begin{equation}\label{E233}
\mathbb{E}\left\vert  I_1(t) \left(Z_{}^V(t)\right)^{-1} \right\vert\le c_\varepsilon t^{- \frac{\kappa}{\alpha} +1-\varepsilon} \le c_\varepsilon
 t^{-\frac{1}{\alpha}}, \qquad t\in [0,T].
\end{equation}
 Let us consider $I_2(t)$.
  By the isometry formula,  Lemma \ref{L6} and using \eqref{E51} we find
 \begin{equation} \label{w221}
\begin{aligned}
&\mathbb{E}\, |  { I_2(t)}  (Z^V(t))^{-1}|
\\ 
&\quad \le C_6\left(  \int_0^t \int_{\{ \vert\xi\vert \le t^{1/\alpha} \}}   \left\vert \xi\right\vert ^{2\kappa-2} \left\vert \xi\right\vert ^{-1-\alpha} \d s\d \xi \right)^{1/2}t^{-\frac{\kappa}{\alpha}}
\\ 
&\quad \le C_7 t^{\frac{2\kappa -2}{2\alpha} - \frac {\kappa}{\alpha}}= C_7t^{-\frac{1}{\alpha}},
\end{aligned}
\end{equation}
 which completes the proof of \eqref{sdd1}. Now we are showing that
 \begin{equation}\label{sdd}
 \mathbb{E} \left\vert  \frac{D Z^{V}{ (t)}}{\left(Z^V(t)\right)^2}\right\vert \le C_8  t^{- \frac{1}{\alpha}}.
  \end{equation}
To this end note that
\begin{align*}
 &\left\vert \int_0^t \int_{\{\delta/2<\vert \xi\vert \le \delta\}}\phi_{\delta}(\xi)  \phi_\delta'(\xi)\Pi (\d s, \d \xi)\right\vert  \\
 &\le \left[ \int_0^t \int_{\{\delta/2<\vert \xi\vert \le \delta\}}\phi_\delta^2(\xi)  \Pi (\d s, \d \xi)\right]^{1/2}
\left[\int_0^t \int_{\{\delta/2<\vert \xi\vert \le \delta\}}\left(  \phi_\delta'(\xi)\right)^2 \Pi (\d s, \d \xi)\right]^{1/2}  \\
& \le  \int_0^t \int_{\{\delta/2<\vert \xi\vert \le \delta\}}\phi_\delta(\xi)  \Pi (\d s, \d \xi)
\int_0^t \int_{\{\delta/2<\vert \xi\vert \le \delta\}}\left\vert  \phi_\delta'(\xi)\right\vert \Pi (\d s, \d \xi)\\
&\le Z^V(t) \int_0^t \int_{\{\delta/2<\vert \xi\vert \le \delta\}}\left\vert  \phi_\delta'(\xi)\right\vert \Pi (\d s, \d \xi).
\end{align*}
Hence, as the arguments from  the derivation of \eqref{E233} we obtain
\begin{align*}
\mathbb{E}\left\vert \frac{{ \int_0^t \int_{\{ \delta /2< |\xi| <\delta\}} } \phi_\delta(\xi)\phi_\delta'(\xi)\Pi ( \d s, \d \xi)} { \left(Z^V(t)\right)^2 }       \right\vert
&\le \mathbb{E}\, \frac{{ \int_0^t \int_{\{ \delta /2< |\xi| <\delta\}} } \left\vert \phi_\delta'(\xi)\right\vert \Pi ( \d s, \d \xi)} { Z^V(t) }\\
&\le C_9t^{-\frac{\kappa}{\alpha} + 1-\varepsilon}\le C_{10} t^{-\frac 1\alpha}.
\end{align*}
Set
 $$
K(t) :=  (Z^V(t))^{-2}\int_0^t \int_{\{ \delta/2\ge |\xi| >  t^{1/\alpha}\} }\, \phi_{\delta}'(\xi)\phi_{\delta}(\xi)  \, \Pi(\d  s, \d  \xi)
$$
and
  $$
H(t) :=  (Z^V(t))^{-2}\int_0^t\int_{\{ |\xi| \le  t^{1/\alpha}\}}\, \phi_{\delta}'(\xi)\phi_{\delta}(\xi) \, \Pi(\d  s, \d  \xi).
$$
Since  $\phi_\delta(\xi)=\vert \xi\vert ^\kappa$ if $\vert \xi\vert \le \delta/2$,  we have
\begin{align*}
\left\vert K(t) \right\vert &\le \kappa \left(Z^V(t)\right)^{-2}   \int_0^t\int_{\{ \delta/2\ge |\xi| >  t^{1/\alpha}\} }\phi_\delta^2(\xi)|\xi|^{-1} \, \Pi(\d  s, \d  \xi) \\
&\le \kappa \left(Z^V(t)\right)^{-2} \,  t^{-\frac{1}{\alpha}}\,  \int_0^t\int_{ \mathbb{R}} \phi_\delta^2(\xi)  \, \Pi(\d  s, \d  \xi)  \\
&= \kappa \left(Z^V(t)\right)^{-2} \,  t^{-\frac{1}{\alpha}}  \left[\left( \int_0^t\int_{ \mathbb{R}} \phi_\delta^2(\xi) \, \Pi(\d  s, \d  \xi)\right)^{1/2}\right]^2 \\
&= \kappa \left(Z^V(t)\right)^{-2} \, t^{-\frac{1}{\alpha}}  \left[\int_0^t\int_{ \mathbb{R}} \phi_\delta(\xi)  \, \Pi(\d  s, \d  \xi)\right]^2= \kappa\,   t^{-\frac{1}{\alpha}}.
\end{align*}

We are dealing now with $H(t)$.   Since
\begin{align*}
&\int_0^t\int_{\{ \vert \xi\vert  \le  t^{1/\alpha}\}}\, |\phi_{\delta}'(\xi)\phi_{\delta}(\xi)| \Pi(\d  s, \d  \xi) \\
&\quad \le
 \left(\int_0^t\int_{ \{ |\xi| \le   t^{1/\alpha}\} }\,  {\left(\phi_\delta'(\xi)\right)^2}\Pi(\d  s, \d  \xi) \right)^{1/2}
\left(\int_0^t \int_{ \mathbb {R} }\,  {\phi_\delta^2(\xi)} \Pi(\d  s, \d  \xi)\right)^{1/2}
\\
&\quad \le Z^V(t) \int_0^t\int_{\{ |\xi| \le   t^{1/\alpha}\} }\left\vert {\phi_\delta'(\xi)}\right\vert \Pi(\d  s, \d  \xi)
= \kappa  Z^V(t) \int_0^t\int_{\{ |\xi| \le   t^{1/\alpha}\} }\left\vert \xi\right\vert^{\kappa -1} \Pi(\d  s, \d  \xi),
\end{align*}
we have, arguing as in  \eqref{w221}, using again  \eqref{E51},
\begin{align*}
 {\mathbb E} \left\vert H(t)\right\vert  &\le \kappa \, {\mathbb E} \, \frac{\int_0^t\int_{\{ |\xi| \le  t^{1/\alpha}\}}\left\vert  \xi\right\vert^{\kappa -1}   \Pi(\d  s, \d  \xi)} {Z^V(t)} \le C_{12} t^{-\frac{1}{\alpha}},
\end{align*}
which finishes  the proof of \eqref{sdd}.
\end{proof}


\section*{Acknowledgment}
We would like to  thank prof. Jerzy Zabczyk for very useful discussions on the topic.

\end{document}